\documentclass[11pt,a4paper]{amsart}
\usepackage[utf8]{inputenc}
\usepackage[T1]{fontenc}
\usepackage{amssymb,amsthm,amsmath}

\usepackage{pgf,tikz}
\usepackage{pgfplots}
\usetikzlibrary{shapes.geometric, arrows}
\usepackage{url}
\usepackage{hyperref}
\usepackage{enumerate}
\usepackage{enumitem}
\usepackage[a4paper]{geometry}
\newtheorem{theorem}{Theorem}[section]
\newtheorem{lemma}[theorem]{Lemma}
\newtheorem{proposition}[theorem]{Proposition}

\newtheorem{definition}[theorem]{Definition}

\theoremstyle{definition}

\usepackage{pdfsync}
\title[Free L\'evy]{ Nonlinear free L\'evy-Khinchine formula and conformal mapping }
\author{Philippe Biane }
\address{Institut Gaspard-Monge, universit\'e Paris-Est Marne-la-Vall\'ee,
5 Boulevard Descartes, Champs-sur-Marne, 77454, Marne-la-Vall\'ee cedex 2,
France}

\DeclareMathOperator\Arg{Arg}
\def\cM{\mathcal M}
\def\cK{\mathcal K}
\def\cL{\mathcal L}

\def\T{\bf T}
\def\R{\bf R}
\def\C{\bf C}
\def\D{\bf D}

\begin{document}

\begin{abstract}
There are two natural notions of L\'evy processes in free probability: the first one has free increments with homogeneous distributions and the other has homogeneous transition probabilities \cite{B}. In the two cases  one can associate a Nevanlinna function to a free L\'evy process. The Nevanlinna functions appearing in the first notion were characterised by Bercovici and Voiculescu \cite{BV1}. I give an explicit parametrisation for the Nevanlinna functions associated with the second kind of free L\'evy processes. This gives a nonlinear
free L\'evy-Khinchine formula.
\end{abstract}

\maketitle

\section {Introduction}
The  convolution of two probability measures on the real line, $\lambda$ and $\mu$, is characterised by 
\begin{equation}\label{conv}
\int_{\bf R} f(x)\lambda*\mu(dx)=E[f(X+Y)]
\end{equation}
 for bounded continuous functions $f$, 
where $X$ and $Y$ are independent random variables, distributed as $\lambda$ and $\mu$.
A probability  distribution $\mu$ is called infinitely divisible if, for every integer $n>0$, it can be written as a convolution power
$\mu=(\mu_{1/n})^{*n}$, for some probability distribution $\mu_{1/n}$. The L\'evy-Khinchine formula gives an integral representation of the logarithm of the Fourier transform of an infinitely divisible distribution i.e. 
$$\int_{\bf R} e^{ivx}\mu(dx)=e^{\theta_\mu(v)}$$
where 
\begin{equation}\label{Lev-Kh}
\theta_\mu(v)=imv+\int_{\bf R}\frac{e^{ivy}-1-ivy}{y^2}(1+y^2)\nu(dy)
\end{equation}
 for some real number $m$ and a positive finite measure $\nu$ on $\bf R$ (the function under the integral  being extended by continuity to $y=0$).
As a consequence, there exists a convolution semigroup of measures $(\mu_t)_{t\geq 0}$ satisfying
$\mu_t*\mu_s=\mu_{s+t}$ and $\int e^{ivx}\mu_t(dx)=e^{-t\theta_\mu(v)}$.
The formula (\ref{Lev-Kh}) is an instance of Choquet's integral representation theorem for convex cones and the extreme cases
correspond to Dirac measures (for $\mu=0$), Gaussian measures ($\mu=\delta_0$) or Poisson distributions ($\mu=\delta_t,t\ne 0$).

The free convolution of two probability measures on the real line, $\lambda$ and $\mu$, defined by Voiculescu  \cite{V}, \cite{BV2}, is characterised by 
\begin{equation}\label{freeconv}
\int_{\bf R} f(x)\lambda\boxplus\mu(dx)=\tau(f(X+Y))
\end{equation}
 for bounded continuous functions $f$, 
where $X$ and $Y$ are free elements in some noncommutative probability space $(\mathcal A,\tau)$, distributed as $\lambda$ and $\mu$, see section \ref{sec_freeconv} below. The free convolution of measures can be computed using their Voiculescu transforms, which are analytic functions  defined on a domain inside the complex upper half plane, moreover one can develop a theory of free convolution which parallels the classical theory of convolution of measures and sums of independent random variables on the real line. In particular there are analogues of the Gauss and Poisson distributions as well as a notion of freely infinitely divisible distributions and a free analogue of the L\'evy-Khinchine formula \cite{BV1}. This formula  reduces to the integral representation formula for Nevanlinna functions defined for $z$ in the upper half plane, indeed the Voiculescu transform of a freely infinitely divisible measure can be expressed as:
\begin{equation}
\varphi(z)=\alpha+\int_{\bf R}{1+xz\over z-x}\nu(dx)
\end{equation} as I recall below in section  \ref{freeinfdiv}. This is the free analogue of the L\'evy-Khinchine formula (\ref{Lev-Kh}).
Associated to a  freely infinitely divisible distribution $\mu$ there is a free convolution semigroup of probability measures  $(\mu_t)_{t\geq 0}$, indexed by real times $t$, 
 satisfying
$$\mu=\mu_1;\quad \mu_s\boxplus\mu_{t}=\mu_{s+t} \ \text{for $s,t\geq 0$.}$$

For $X,Y$ as in  (\ref{freeconv})
there exists a  Markov kernel $p_{\lambda,\mu}(x,du)$ on $\bf R$ such that, for any bounded continuous functions $f,g$ one has
\begin{equation}\label{freemarkov}
\tau(f(X)g(X+Y))=\int_{\bf R}\left(\int_{\bf R} g(u)p_{\lambda,\mu}(x,du)\right)f(x)\lambda(dx).
\end{equation}
This  is   analogous   to the classical situation where $X$ and $Y$ are independent random variables: in this case one has
\begin{equation}
E(f(X)g(X+Y))=\int_{\bf R}\left(\int_{\bf R} g(u)q_{\lambda,\mu}(x,du)\right)f(x)\lambda(dx)
\end{equation}
where  the kernel is given by 
$q_{\lambda,\mu}(x,du)=(\mu*\delta_x)(du)$, the translate of $\mu$ by $x$, in particular it does not depend on $\lambda$. 

The Markov kernel $p_{\lambda,\mu}$ can be  computed in terms of the Cauchy transforms of the measures and this leads to a subordination property of these Cauchy transforms (cf \cite{B}), which I recall in section \ref{sec_subor} below. 
If $(\mu_t)_{t\geq 0}$ is a convolution semigroup of freely infinitely divisible distributions, then one can define accordingly Markov kernels $\cK_{s,t}$, for $s<t$,
corresponding to the convolution equations
$$\mu_s\boxplus\mu_{t-s}=\mu_t.$$
These kernels satisfy the Chapman-Kolmogorov equation:
$$
\cK_{s,t}\circ \cK_{t,u}=\cK_{s,u}\quad\text{for $s<t<u$}.
$$
Contrary to the case of classical convolution, the homogeneity of the increments  does not imply that the kernels are time-homogeneous, i.e. in general $\cK_{s,t}$ does not depend only on $t-s$. 
The question therefore arises of finding continuous families of measures 
\begin{equation}\label{hommeas}
\mu_t, \ \text{for}\ t\geq 0,\quad \mu_{s,t}, \ \text{for}\  s<t, \quad\text{such that}\quad
\mu_s\boxplus\mu_{s,t}=\mu_t
\end{equation} and such that the corresponding kernels $\cK_{s,t}$ depend only on $t-s$. In \cite{B} I gave a necessary and sufficient condition for a family such as (\ref{hommeas}) to correspond to a time-homogeneous transition kernel. However these conditions are hard to check and  I did not give an explicit description of all the solutions. The purpose of this paper is to give an answer to this question and in particular to give a parametrisation of all solutions, which we will call the nonlinear free L\'evy-Khinchine formula. This parametrisation has a strong geometric flavour and uses in an essential way the theory of conformal mappings of the upper half plane.

I will also consider the case of free multiplicative convolution, for which analogous results can be obtained.

This paper is organised as follows. In section 2, I recall the necessary facts from complex analysis: Cauchy transforms, Nevanlinna functions, and from free probability: free convolution, Voiculescu transform, free infinitely divisible distributions,  subordination
and the Markov property. I also state the main problem that is solved in the paper, which is to characterise free additive L\'evy functions of the second kind. In section 3, I consider primitives of Nevanlinna functions and investigate their behaviour as conformal mapping, in particular I give necessary and sufficient conditions for such functions to map the upper half-plane to a domain containing the upper half-plane, which is the crucial property needed later.
In section 4, I solve the main problem by providing an explicit characterisation of the free additive L\'evy functions of the second kind.
Finally, the case of multiplicative convolution is discussed in section 5.
\section{Preliminaries}
\subsection{Some tools from complex analysis}

\subsubsection{Cauchy and Voiculescu transforms}

The Cauchy transform of a probability measure
$\mu$ on $\bf R$ is given by 
$$G_{\mu}(\zeta)=\int_{\bf R}{1\over \zeta-u}\mu(du),$$
which  defines an analytic function of $\zeta\in {\bf C}\setminus \bf R$, such that
$G_{\mu}(\bar \zeta)=\overline{ G_\mu( \zeta)}$ and
$G_{\mu}({\bf C}^+)\subset {\bf C}^-$, where  ${\bf C}^\pm$ denote the upper and lower half planes i.e. ${\bf C}^+=\{z\in{\bf C}|\Im(z)>0\}, {\bf C}^-=-{\bf C}^+$. This function uniquely determines the measure
$\mu$. For 
$\alpha,\beta>0$, let $$\Theta_{\alpha,\beta}=
\{z=x+iy\,\vert \,y<0;\alpha y< x<-\alpha y;\vert z\vert\leq
\beta\}.$$
For every $\alpha>0$, there exists a real number
$\beta>0$ such that  the function
$G_{\mu}$ has a right inverse defined on the domain 
$\Theta_{\alpha,\beta}$, taking values in some domain of the form
$$\Gamma_{\gamma,\lambda}=\{z=x+iy\,\vert\, y>0;-\gamma y< x<\gamma
y;\vert z\vert\geq
\lambda\}$$
with $\gamma,\lambda>0 $.  Call
$K_{\mu}$ this right inverse, and let
$R_{\mu}(z)=K_{\mu}(z)-{1\over z}$. We shall also need the notations

\begin{equation}\label{cauchy1}
F_{\mu}(\zeta)={1\over G_{\mu}(\zeta)}
\end{equation} and
\begin{equation}\label{cauchy2}
\varphi_{\mu}(z)=R_{\mu}({1\over z})=F_{\mu}^{-1}(z)-z
\end{equation} where
$F_{\mu}^{-1}$ is defined in some domain of the form
$\Gamma_{\alpha,\beta}$. The function $\varphi_\mu$ is defined on the same domain as $F_\mu^{-1}$ and takes its values in
${\bf C}^-\cup\bf R$. It is called the Voiculescu transform of $\mu$.
\subsubsection{Nevanlinna functions}\label{sec_nevanlinna}

An analytic function $\varphi$, defined on ${\bf C}^+$, with values in ${\bf C}^-\cup\bf R$, is called a Nevanlinna function.
The Nevanlinna
representation  gives  real
numbers
$\alpha\leq 0,\,\beta$ and a finite positive  measure $\nu$,  on $\bf R$, such that 
\begin{equation}\label{nevanlinna}
\varphi(z)=\alpha z+\beta +\int_{\bf R}\frac{1+uz}{z-u}\nu(du).
\end{equation}
The measure $\nu$ can be recovered from $\varphi$ by
\begin{equation}\label{nevanilnna2}
\nu(du)=\lim_{\varepsilon\to 0}\frac{-\Im(\varphi(u+i\varepsilon))}{2\pi(1+u^2)}du
\end{equation}
while 
\begin{equation}\label{nevanilnna3}
\alpha=\lim_{v\to +\infty}\frac{\varphi(iv)}{iv}.
\end{equation}
Finally
$$\varphi(i)=\alpha i+\beta-i\int_{\bf R}\nu(du)$$ allows to revover all parameters.

Observe that, if $\varphi$ takes a real value at some point, then it is constant, as follows from the maximum principle.
Also the extreme points in the integral representation (\ref{nevanlinna}) correspond to the maps
$z\mapsto  \frac{1+xz}{z-x}$, which are conformal mappings from the upper half plane onto itself.

Finally we note that, if $\int |u|\nu(du)<+\infty$, then 
$\int_{-\infty}^{+\infty}\frac{1+uz}{z-u}\nu(du)\to \int u\,\nu(du)$ if $|z|\to\infty$ with $z$ in some domain $\Gamma_{\gamma,\lambda}$.

\subsection{Free convolution and freely indivisible distributions}
\subsubsection{Free convolution }\label{sec_freeconv}
We recall the definition of the free convolution of measures and how it can be computed, 
 see e.g.
\cite{BV2} for these results.
Let
$\lambda$ and $\mu$ be probability measures on $\bf R$, then there exists a
non-commutative  probability space $(A,\tau)$ and self-adjoint elements $X,Y$
affiliated to $A$, with respective   distributions $\lambda$ and $\mu$, such
that  $X$ and $Y$ are free, i.e. the von Neumann algebras generated by their
spectral projections are free. The distribution of
$X+Y$ depends only on
$\lambda$ and $\mu$, it is called the free additive convolution of $\lambda$ and
$\mu$  and is denoted by
$\lambda\boxplus
\mu$. This  defines a symmetric and associative binary operation on the set
of probability measures on $\bf R$.  The  free
additive convolution is linearised by the Voiculescu transform (\ref{cauchy2}), indeed one has
$$\varphi_{\lambda\boxplus \mu}=\varphi_{\lambda}+\varphi_{\mu}$$ on some domain of
the
form
$\Gamma_{\alpha,\beta}$ where these three functions are defined. Since
$\lambda\boxplus
\mu$ is determined by the restriction of 
$\varphi_{\lambda\boxplus
\mu}$ to one of these domains, this characterises completely the measure
$\lambda\boxplus \mu$.

\subsubsection{Processes with free increments }\label{freeincrpro}
Processes with free increments were studied in \cite{B}.
In short, a process with free increments is a family of non-commutative random variables $(X_t)_{t\geq 0}$, in an non-commutative  probability space $(A,\tau)$, such that 
for any $s<t$ the increment $X_t-X_s$ is free with the von Neumann algebra generated by the $(X_u)_{u\leq s}$ (some care is needed when the operators are unbounded and one has to use affiliated subalgebras, see \cite{B} for details).
The laws of the increments $X_t-X_s$, for $s<t$, denoted $\mu_{s,t}$ satisfy the relations 
\begin{equation}\label{freeincr}
\mu_{s,t}\boxplus\mu_{t,u}=\mu_{s,u}\quad \text{ for }\quad s<t<u.
\end{equation} Conversely, given probability distributions 
$\mu_{s,t}$ satisfying  relations (\ref{freeincr}), together with some continuity assumption and an initial distribution $\mu_0$, there exists a noncommutatif process with  free increments
distributed as $\mu_{s,t}$ \cite{B}.
\subsubsection{Free infinitely divisible distributions}\label{freeinfdiv}

There is a notion of infinitely divisible measures for the free additive
convolution: a measure $\mu$ is freely infinitely divisible if for all $n>0$ there exists $\mu_{1/n}$ such that
$\mu=(\mu_{1/n})^{\boxplus n}$.  There is also an analogue of the L\'evy-Khinchine formula, which was obtained in \cite{BV1}.
A probability measure $\mu$ on $\bf R$ is freely infinitely divisible if and only if
its Voiculescu transform $\varphi_\mu$ has
an analytic continuation to the
whole
of ${\bf C}^+$, with values in ${\bf C}^-\cup\bf R$ and one has $$\lim_{v
\rightarrow\infty,v\in\bf R}{\varphi_\mu(iv)\over iv}=0.$$ 
The Nevanlinna representation  (\ref{nevanlinna})  implies that 
\begin{equation}\label{freelevkh}
\varphi_\mu(z)=\beta+\int_{\bf R}{1+uz\over z-u}\nu(du)
\end{equation}
for some positive finite measure $\nu$, called the free L\'evy measure of $\mu$. The formula (\ref{freelevkh})
is the free analogue of the L\'evy-Khinchine formula. It expresses an arbitrary infinitely divisible distribution in terms of the Wigner semi-circle distribution (corresponding to $\nu=\delta_0$), which is the free analogue of the Gauss distribution and the Pastur-Marchenko distributions (for $\nu=\delta_t$ with $t\ne 0$), which are the free Poisson distributions. 
If $\mu$ is freely infinitely divisible
then for all $t\geq 0$  there exists a probability measure on the real line $\mu_t$ 
such that $\varphi_{\mu_t}=t\varphi_{\mu}$ and these measures satisfy the relations
$$\mu_s\boxplus\mu_t=\mu_{s+t}.$$ 
The parallel between classical and free infinitely divisible distributions goes quite far, for example one can find free analogues of the classical theory of stable distributions and domains of attractions, see \cite{BP}.

\subsection{Subordination and the Markov property}\label{sec_subor}
Given probability distributions $\mu,\nu$ on $\bf R$, there 
 exists a subordination relation between the Cauchy transforms of $\mu$ (or $\nu$) and of $\mu\boxplus \nu$. As shown in \cite{B}, this relation expresses the Markov property of the free convolution.
We recall the main theorem of \cite{B}.
\begin{theorem}\label{subor}
 Let $(A,\tau)$ be a
non-commutative probability space, $B$ be a von
Neumann subalgebra of
$A$, let
$Y\in A$
be a self-adjoint element which is free with $B$, and let 
$X\in B$ be self-adjoint. Denote by $\lambda$ and $\mu$  the distributions of
$X$  and $Y$, then there
exists a Feller Markov kernel $\cK=k(x,du)$ on $\bf R\times \bf R$ and an
analytic
function F  on 
${\bf C}\setminus \bf R$ such that  
\begin{enumerate}
\item 
For any  Borel bounded function $f$ on $\bf R$ one has $\tau(f(X+Y)\vert B)=\cK
f(X).$
\item $F(\bar\zeta)=\overline{F(\zeta)}$, \ \  $F({\bf C}^+)\subset {\bf C}^+$.
\item $Im(F(\zeta))\geq Im(\zeta)$  for $\zeta\in{\bf C}^+$.
\item
${F(iy)\over iy}\rightarrow1$ as $y\rightarrow+\infty$, $y\in\bf R$\item
for all
$\zeta\in{\bf C}\setminus
\bf R$ one has
$\int_{\bf R}(\zeta-u)^{-1}k(x,du)=(F(\zeta)-x)^{-1}.$

\item For all $\zeta\in{\bf C}\setminus \bf R$ one has
$G_{\lambda}( F(\zeta))=G_{\lambda\boxplus \mu}(\zeta)$
\end{enumerate}

\medskip

Here $\tau(.|B)$ denotes the conditional expectation, $\cK f(x)=\int_{\bf R} f(u) k(x,du)$
and the map $F$ is uniquely determined by properties $(4)$ and
$(6)$.
\end{theorem}
Property (1) above is the Markov property of free convolution while 
 (6) is the subordination property relating the Cauchy transforms of $\lambda$ and $\lambda\boxplus\mu$.

Using Theorem \ref{subor} we get, for each process with free increments, as in section \ref{freeincrpro}, a family of Markov kernels $\cK_{s,t};s<t$ on the real line, satisfying the Chapman-Kolmogorov relation
\begin{equation}\label{chapkol}
\cK_{s,t}\circ \cK_{t,u}=\cK_{s,u}\quad\text{for $s<t<u$}.
\end{equation}
These kernels are determined, using eq.(5), by analytic functions $F_{s,t}$ mapping the upper halfplane to itself and satisfying 
\begin{equation}\label{chapkolF}
F_{s,t}\circ F_{t,u}=F_{s,u}\quad\text{for $s<t<u$}.
\end{equation}
 We call such a family of kernels {\sl time homogeneous} if $\cK_{s,t}\equiv \cK_{t-s}$ (or equivalently $F_{s,t}\equiv F_{t-s}$) depends only on $t-s$. 
If this is the case then the kernels $\cK_{t}$ form a semigroup (and the maps $F_t$ form a semigroup of analytic maps on ${\bf C}^+$).

As is easily seen on examples, see e$.$g$.$ section 5 of \cite{B}, in general the kernels $\cK_{s,t}$ are not time homogeneous, when the increments are i$.$e$.$ when $\mu_{s,t}\equiv \mu_{t-s}$. It is therefore natural to ask whether there exists processes with non homogeneous free increments and with time homogeneous transition probabilities. In \cite{B} a characterisation was given in the following theorem.
 \begin{theorem}\label{freenllevy}
 Let $\mu_t,t\geq 0$ and 
$(\mu_{s,t})_{s<t\in\bf R_+}$  be families of probability measures satisfying 
\begin{equation}\label{sg}
\mu_s\boxplus\mu_{s,t}=\mu_t;\quad \mu_{s,t}\boxplus\mu_{t,u}=\mu_{s,u}
\end{equation}
for all $s<t<u$. Let 
$(\cK_{s,t})_{s<t\in\bf R_+}$ be the corresponding Markov transition functions on $\bf R$.
Assume that the kernels are time homogeneous, then 
 the kernels
$\cL_{t}\equiv\cK_{0,t}$ for $t\geq 0$, 
 form a Feller Markov
semi-group. Let $F_{s,t}$ be the analytic functions  associated to the
kernels $\cK_{s,t}$ by Theorem \ref{subor}. The maps $F_{t}\equiv F_{0,t}$, where
$F_0$ is the identity function, form a continuous semigroup, under composition, of
analytic transformations of
${\bf C}^+$ and $F_{s,t}=F_{t-s}$, moreover there exists a  
Nevanlinna function $\varphi$ such
that 

\begin{equation}\label{Nevinfini}
\frac{\varphi(\zeta)}{\zeta}\to_{\genfrac{}{}{0pt}{3}{\zeta\to\infty}{\zeta\in\Gamma_{\alpha,\beta}}} 0
\end{equation}
 in every domain of the form
$\Gamma_{\alpha,\beta}$ and such that the maps 
$F_t$, for $t\geq 0$, satisfy the differential equation 
\begin{equation}\label{diff}
\frac{\partial F_t}{
\partial t}+\varphi(F_t)=0,\qquad F_0(z)=z.
\end{equation}

 Conversely, let  $\varphi$ be a Nevanlinna function satisfying
(\ref{Nevinfini}) in some domain of the form
$\Gamma_{\alpha,\beta}$, and let
$(F_t)_{t\in\bf R_+}$ be the semi-group of analytic maps of ${\bf C}^+$ obtained
by solving (\ref{diff}) with initial condition $F_0(z)=z$, then there exists
 $(\mu_t)_{t\geq 0}$ and 
$(\mu_{s,t})_{s<t\in\bf R_+}$  families of probability measures satisfying (\ref{sg}) with associated semi-group of maps $(F_t)_{t\in\bf R_+}$, if and
only if, for every
$t>0$ the function
$\varphi\circ F_t^{-1}\circ F_{\mu_0}^{-1}$  has an analytic continuation to
$\bf C^+$, with values in $\bf C^-$.

\end{theorem}
The Nevanlinna functions having the properties listed in Theorem \ref{freenllevy} have been called {\sl free additive L\'evy functions of the second kind} (or FAL2) in \cite{B}. One can easily check, by explicit computations \cite{B}, that the functions $z\mapsto -z^\rho$ with $0<\rho<1$ are FAL2 functions while the Nevanlinna functions $z\mapsto z^\theta$, for $-1<\theta<0$, are not.

The characterisation of FAL2 functions in this theorem is rather indirect, it  is not easy to check moreover it does not provide a nice parametrisation of the set of FAL2 functions.
In the following we shall show that one can give a more explicit parameterisation these functions, at least in the case $\mu_0=\delta_0$. For this we use properties of primitives of Nevalinna functions, as explained in the next section, as well as classical results  on starlike domains in conformal mapping theory. We call this parametrisation the {\sl nonlinear free L\'evy-Khinchine formula} since, as we shall see the set of FAL2 functions can be parametrised by a convex set, up to some non-linear transformation.
\section{Conformal mappings associated with  Nevanlinna functions}
\subsection{Primitives of Nevalinna functions}
Let $\psi$ be a Nevalinna function and $\Psi=-\int\psi(z)dz$ be a primitive of $-\psi$, which is holomorphic on ${\bf C}^+$.

\begin{lemma}\label{univalent}
If $\psi\ne 0$ then the function  $\Psi$ is univalent on ${\bf C}^+$.
\end{lemma}
\begin{proof}
If $\psi$ is real, then it is constant and $\Psi(z)=az+b$ for some $a\ne 0$ therefore the claim is clear. If not then $\Im(\psi(z))<0$ for all $z\in {\bf C}^+$.
For $z_1\ne z_2$ in ${\bf C}^+$ one has  
$$\frac{\Psi(z_2)-\Psi(z_1)}{z_2-z_1}=-\int_0^1\psi(z_1+t(z_2-z_1))dt$$ therefore 
$\Im(\frac{\Psi(z_2)-\Psi(z_1)}{z_2-z_1})>0$. \end{proof}

\medskip
It follows from Lemma \ref{univalent} that $\Psi$ maps conformally  ${\bf C}^+$ onto some domain $\Omega\subset \bf C$.
The class of domains which are obtained in this way can be characterised by a geometric property, which is the upper half-plane version of a classical result on univalent functions in the unit disk, concerning starlike domains. 
\subsection{Starlike domains}\label{starlike}
\begin{definition}
A domain $\Omega\subset \bf C$ is called {\sl starlike at} $-\infty$ if $\Omega\ne\emptyset,\bf C$ and,  for any $t>0$, one has $\Omega-t\subset \Omega$. 
\end{definition}
A domain $\Omega$, which is starlike at $-\infty$, is a union of open horizontal half-lines
$$D_\Omega(q)=\Omega\cap\{p+iq\,|\, p\in{\bf R}\} =\{p+iq\,|\, p< d_\Omega(q)\}$$ where $d_\Omega:{\bf R}\to[-\infty,+\infty]$ is a lower semicontinuous function and 
$d^{-1}_\Omega(-\infty)=]-\infty,q_-]\cup [q_+,+\infty[$ whith $-\infty\leq q_-<q_+\leq +\infty$.
\begin{proposition}
Let $\psi$ be a nonzero Nevanlinna function and $\Psi$ be a primitive of $-\psi$ then the domain $\Psi({\bf C}^+)$ is starlike at $-\infty$. Conversely, for any
 domain $\Omega$, starlike at $-\infty$, there exists $\psi$, a nonzero Nevanlinna function and $\Psi$ be a primitive of $-\psi$, such that
$\Omega=\Psi( {\bf C}^+)$.
\end{proposition}
The proof is similar to the case of univalent functions on the unit disk, cf Pommerenke Ch. 2.2 \cite{P}.

It is instructive to  consider the case of rational Nevanlinna functions. Let $\psi(z)$ be such a function, with  partial fraction expansion
$$\psi(z)=az+b+\sum_{k=1}^N\frac{\alpha_k}{z-\xi_k}$$
where $a< 0$, $b$ is real,  the $\xi_k$ are real (with $\xi_1<\xi_2<\ldots<\xi_N$) and the $\alpha_k$ are positive.
We have
$$\Psi(z)=-\frac{1}{2}az^2-bz-\sum_k\alpha_k\log(z-\xi_k)$$ where we take the determination of the logarithm on $\bf C\setminus \bf R_-$ such that
 $\log(t)>0$ for $t>0$. The map $\Psi$ extends continuously (even analytically) to the boundary of $\bf C^+$ (i.e. to $\bf R$) except at the points $\xi_k$, moreover its imaginary part is constant on each interval $]\xi_k,\xi_{k+1}[$, while its real part is a stricly convex function on each of these intervals, whith limit $+\infty$ at each boundary point. It follows that
the image  of  ${\bf C}^+$ by $\Psi$ is the complement of a sequence of horizontal half-lines 
${\mathcal D}_1,{\mathcal D}_2,\ldots,
{\mathcal D}_{N+1}$, each of the form ${\mathcal D}_j=\{p+iq_j\,|\,p\geq p_j\}$ 
 hence
\begin{equation}\label{starshaped}
\Psi({\bf C}^+)=\Omega={\bf C}\setminus ({\mathcal D}_1\cup {\mathcal D}_2\cup\ldots\cup
{\mathcal D}_{N+1}).
\end{equation}
Conversely, it is not difficult to check that for any finite family of horizontal half-lines, as above, the conformal map from 
${\bf C}^+$ to ${\bf C}\setminus ({\mathcal D}_1\cup {\mathcal D}_2\cup\ldots\cup
{\mathcal D}_{N+1})$, mapping $\infty$ to $\infty$, is the primitive of the opposite of a rational Nevalinna function. 

As an example, the conformal mapping $\Psi(z)=z^2/2-\log(z)$, corresponding to $\psi(z)=-z+\frac{1}{z}$, maps  $\bf C^+$ to $\bf C\setminus ({\mathcal D}_1\cup{\mathcal D_2})$ where 
 ${\mathcal D_1}=\{p\,|\,p\geq 1/2\}$ and ${\mathcal D_2}=\{p-i\pi\,|\,p\geq 1/2\}$ are  two horizontal half lines. The figure below shows some of 
the flow lines and equipotential lines, i.e. the images by $\Psi$ of  the lines  $\Im(z),\Re(z)=cst$ in  ${\bf C}^+$.
\begin{figure}[h!]
  
  \includegraphics[width=16cm]{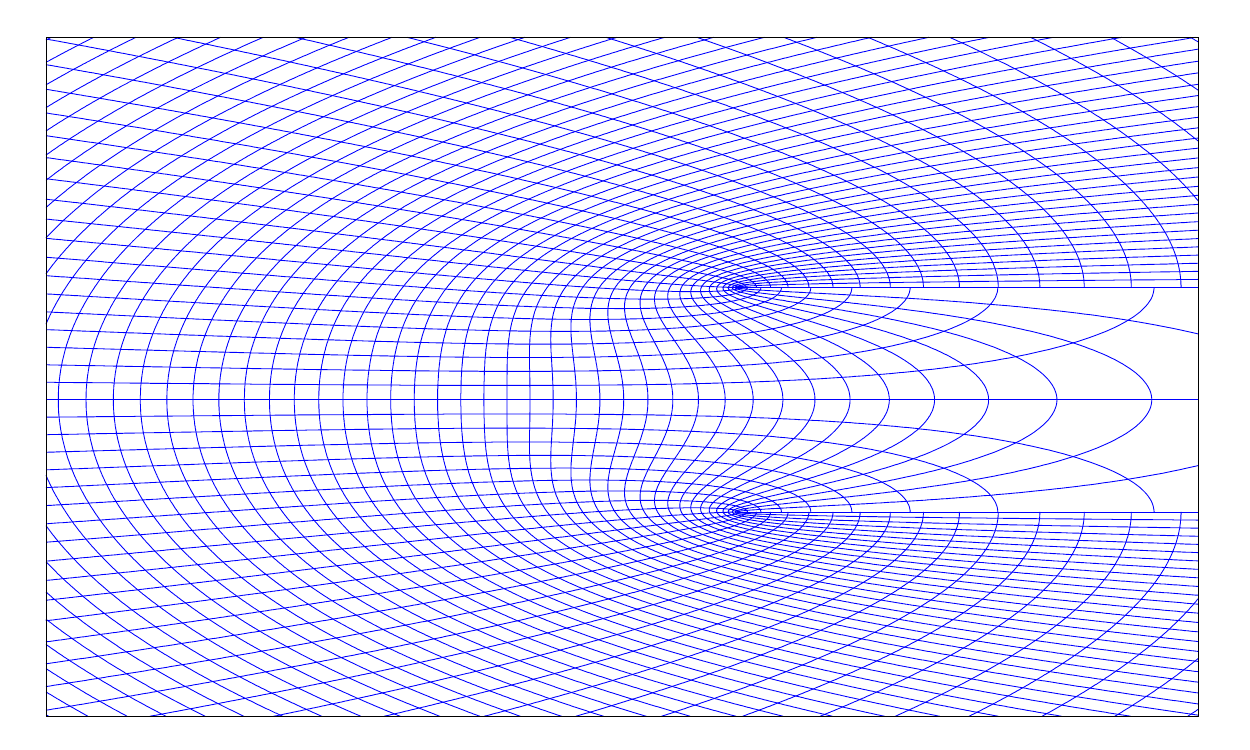}
\caption{ The image of $\Psi(z)=z^2/2-\log(z)$}
\end{figure}

\subsection{Starlike domains containing an upper halfplane}\label{starup}
In this section, $\psi$ denotes a Nevanlinna function with canonical representation
$$
\psi(z)=\alpha z+\beta +\int_{-\infty}^{+\infty}\frac{1+uz}{z-u}\nu(du)
$$
and $\Psi$ a primitive of $-\psi$. We look for conditions on $\psi$ ensuring that the starlike domain  $\Psi(\bf C^+)$ contains a translate of 
$\bf C^+$.
\begin{lemma}\label{C+} If $\Psi(\bf C^+)$ contains a translate of 
$\bf C^+$ then one has $\int_0^\infty u^2\nu(du)<+\infty$.
\end{lemma}
\begin{proof}
The region ${\bf C}^+_\varepsilon=\{z\,|\,\Im(z)>\varepsilon\}$ is mapped by  $\Psi$ to the region on the left of the curve $x\mapsto \Psi(x+i\varepsilon)$ where the function $\Psi(x+i\varepsilon)$   has strictly increasing imaginary part. The image of $\Psi$ is a proper domain in $\bf C$ therefore  there exists a point $w\in\bf R$ such that $\Im(\Psi(w+i\varepsilon))$ remains bounded as $\varepsilon\to 0$. 
One has
$$
\begin{aligned}\label{ImPsi}
\Im(\Psi(A+i\varepsilon)-\Psi(w+i\varepsilon))&=\int_w^A\Im(\psi(x+i\varepsilon))dx
\\
&=-\varepsilon \alpha (A-w)+\int_w^A\left(\int_{-\infty}^{\infty}\frac{\varepsilon(1+u^2)}{(x-u)^2+\varepsilon^2}\nu(du)\right)dx
\end{aligned}
$$
One can easily see that $$\int_w^A\left(\int_{-\infty}^{\infty}
\frac{\varepsilon(1+u^2)}{(x-u)^2+\varepsilon^2}\nu(du)\right)dx\to_{\varepsilon\to 0}
 \int_w^A(1+u^2)\nu(du).$$   
If $\Psi(\bf C^+)$ contains a translate of 
$\bf C^+$ then this quantity must remain bounded as $A\to\infty$, therefore $\int_0^\infty u^2\,\nu(du)<+\infty$.\end{proof}
From the description of starlike domains in section \ref{starlike} we see that, if $\int_0^\infty u^2\nu(du)<+\infty$ then there exists  some real number $q_0$ such that
either $d_\Omega(q)=+\infty$ for $q\geq q_0$ or $d_\Omega(q)=-\infty$ for $q\geq q_0$.
The domain $\Psi(\bf C^+)$ contains a translate of 
$\bf C^+$ if and only if 
the first case holds. It follows that, if $\Psi(\bf C^+)$ contains complex numbers of arbitrarily large imaginary part then
$\Psi(\bf C^+)$  contains a translate of $\bf C^+$.
Note that, by Cauchy-Schwarz inequality one has $\int_0^\infty u\,\nu(du)<+\infty$ and the quantity $\int_{\bf R} u\,\nu(du)$ is well defined in $[-\infty,+\infty[$.
 \begin{lemma}\label{NevC+}
Assume that $\int_0^\infty u^2\,\nu(du)<+\infty$.
\begin{enumerate}

\item
If 
$\alpha<0$ then $\Psi(\bf C^+)$ contains a translate of 
$\bf C^+$.
\item
If 
$\alpha=0$ and $\beta+\int_{\bf R} u\,\nu(du)<0$ then $\Psi(\bf C^+)$ contains a translate of 
$\bf C^+$.
\item
If 
$\alpha=0$ and $\beta+\int_{\bf R} u\,\nu(du)\geq 0$ then $\Psi(\bf C^+)$ does not contain a translate of 
$\bf C^+$.
\end{enumerate}
\end{lemma}
\begin{proof}\

 \begin{enumerate}

\item One has $\Psi(ye^{i\pi/4})\sim \frac{-\alpha}{2}iy^2$ for large $y$ therefore 
$\Psi(\bf C^+)$ contains complex numbers of arbitrarily large imaginary part and one concludes from the discussion before Lemma \ref{NevC+}.
\item
If $\alpha=0$ and $\int_{\bf R}u\,\nu(du)<0$ then 
$\Re(\psi(iy))<-\varepsilon$ for $y$ large enough and some $\varepsilon>0$. It follows that $\Im(\Psi(iy))\to\infty$ as $y\to\infty$ and we conclude by the same argument. \item
If $\alpha=0$ and $\int_{\bf R}u\,\nu(du)=\gamma\geq 0$ then 
$$\psi(z)=\gamma z+\int_{\bf R}\frac{1+u^2}{z-u}\nu(du)$$
and one has
$$
\begin{aligned}
\Im(\Psi(x+iy)-\Psi(x+i))&=-\gamma(y-1)-\int_1^y\left[\int_{\bf R}\frac{(1+u^2)(x-u)}{(x-u)^2+w^2}\nu(du)\right]dw
\\ &\leq 
-\gamma(y-1)+\int_1^y\left[\int_x^{\infty}\frac{(1+u^2)(u-x)}{(x-u)^2+w^2}\nu(du)\right]dw
\end{aligned}$$
One has, for $y\geq 0$, 
\begin{equation}\label{arctan}
\int_1^y \frac{v}{v^2+w^2}dw\leq\frac{\pi}{2}
\end{equation}
 therefore
\begin{equation}\label{imag}
\Im(\Psi(x+iy)-\Psi(x+i))\leq \gamma(1-y)+\frac{\pi}{2}\int_0^\infty (1+u^2)\nu(du)
\end{equation}
Since  $\Im(\Psi(x+i))$ is  uniformly bounded in $x$ it follows that  $\Im(\Psi(z))$ is uniformly bounded on ${\bf C}^+$.
\end{enumerate}
\end{proof}
Using Lemmas \ref{C+} and \ref{NevC+} we can now state  necessary and sufficient conditions on $\psi$ so that  $\Psi(\bf C^+)$  contains a translate of 
$\bf C^+$.

\begin{proposition}
Let  $\psi$ be a Nevanlinna function with
canonical representation 
$$
\psi(z)=\alpha z+\beta +\int_{-\infty}^{+\infty}\frac{1+uz}{z-u}\nu(du)
$$
and $\Psi$ be a primitive of $-\psi$, then $\Psi({\bf C}^+)$ contains a translate of ${\bf C}^+$ if and only if one of the following exclusive conditions is fulfilled:
\begin{enumerate}
\item
$\int_0^\infty u^2\nu(du)<+\infty$ and  $\alpha<0$.
\item
$\int_0^\infty u^2\nu(du)<+\infty,\, \alpha=0$ and $\beta+\int_{\bf R} u\,\nu(du)<0$.
\end{enumerate}
\end{proposition}

Here are two examples illustrating the different situations.

Figure 2 shows $\psi(z)=-z^{1/2}$ where $\alpha=0$ and  $\nu(du)=\frac{\sqrt{-u}du}{2\pi(1+u^2)}1_{u<0}$, with  
$ \beta+\int u\,\nu(du)=-\infty$. One has 
$\Psi(z)=\frac{2}{3}z^{3/2}$ and   the image $\Psi(\bf C^+)$ is a $3/4$ plane, which contains the upper half-plane. 

\begin{figure}[h!]
  
  \includegraphics[width=12cm]{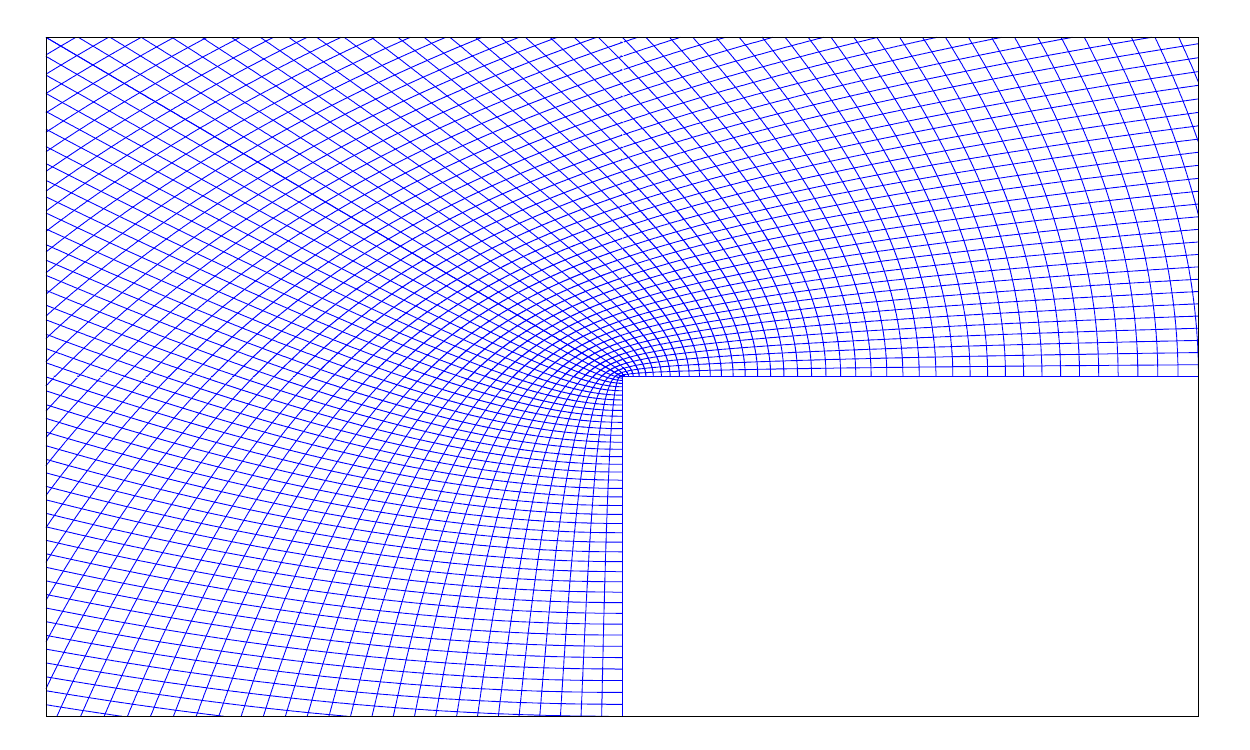}
\caption{ The image of $\Psi(z)=\frac{2}{3}z^{3/2}$}
\end{figure}

In Figure 3 one has 
$\psi(z)=z^{-1/2}$ where $\alpha=0$ and $\nu(du)=\frac{du}{2\pi\sqrt{-u}(1+u^2)}1_{u<0}$ while  $ \beta+\int u\,\nu(du)=0$ and 
$\Psi(z)=-2z^{1/2}$. This time the   image $\Psi(\bf C^+)$ is a $1/4$ plane, it  does not contain a translate of the upper half-plane. 
\begin{figure}[h!]
  
  \includegraphics[width=12cm]{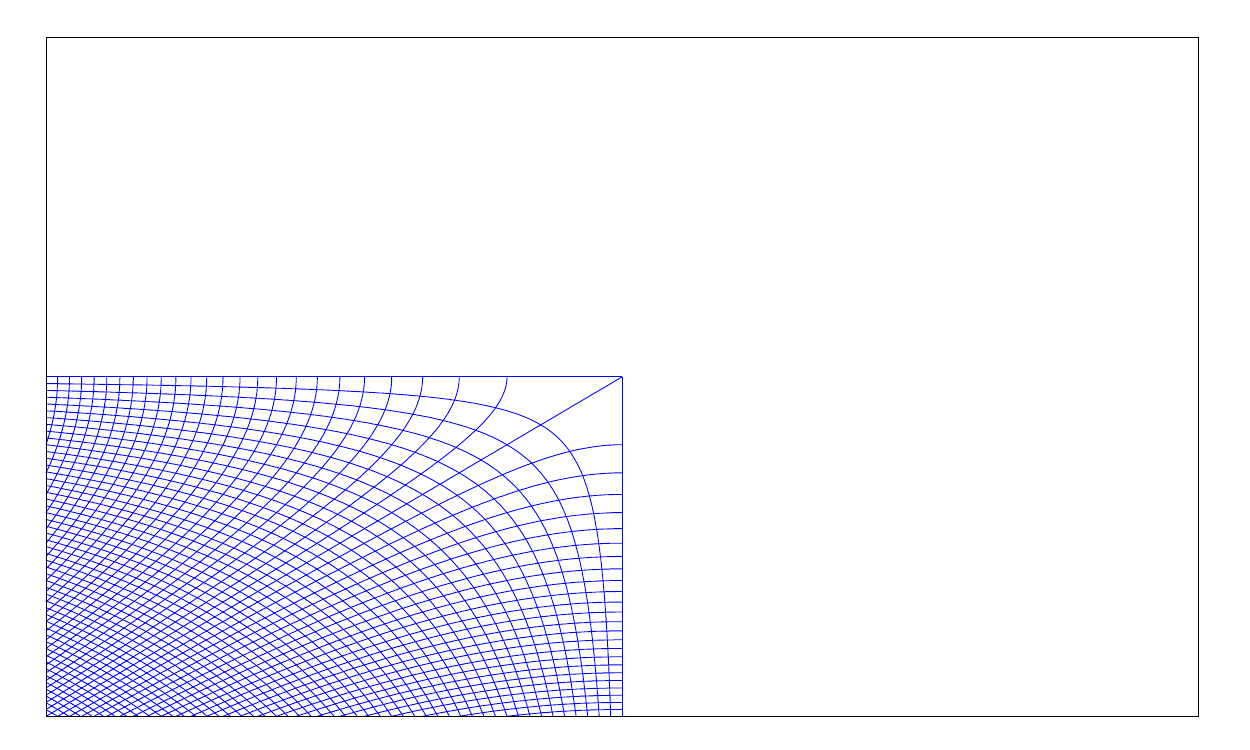}
\caption{ The image of $\Psi(z)=-2z^{1/2}$}
\end{figure}

\section  {Free L\'evy processes with homogeneous transition probabilities}
\subsection{Some preliminary computations}
Let us recall that, by Theorem \ref{freenllevy} we are trying to characterise Nevanlinna functions $\varphi$
 such
that 
\begin{enumerate}[label=(\roman*)]
\item
$
\frac{\varphi(\zeta)}{\zeta}\to_{\genfrac{}{}{0pt}{3}{\zeta\to\infty}{\zeta\in\Gamma_{\alpha,\beta}}} 0
$
 in every domain of the form
$\Gamma_{\alpha,\beta}$
\item For any $t\geq 0$,
$\varphi\circ F_t^{-1}$  has an analytic continuation to
$\bf C^+$, with values in $\bf C^-$.
\end{enumerate}
 Here $F_t(z)$, for $t\geq 0$ is
the semi-group of analytic maps of ${\bf C}^+$ obtained
by solving 
\begin{equation}\label{diff1}
\frac{\partial F_t}{
\partial t}+\varphi(F_t)=0 
\end{equation}
 with initial condition $F_0(z)=z$,

In order to solve (\ref{diff1}) it is natural to introduce a primitive of $-1/\varphi$, denoted  $\Phi$. Indeed  $\Phi$, being the primitive of a Nevanlinna function, is such that $\Phi({\bf C}^+)+t\subset
\Phi({\bf C}^+)$ for $t\geq 0$ and  one has $F_t(z)=\Psi(\Phi(z)+t)$ where $\Psi$ is the inverse of the conformal mapping $\Phi$. Note that  one has
$\Psi'=-\varphi\circ \Phi$.
\subsection{The free nonlinear L\'evy-Khinchine formula}
\subsubsection{}\label{direct}
We denote by  $\psi$  a Nevanlinna function
and $\Psi$ a primitive of $-\psi$,
 such that $\Psi(\bf C^+)$ contains 
$\bf C^+$, as characterised in section \ref{starup}. Let $\Phi=\Psi^{-1}:\Psi({\bf C}^+)\to {\bf C}^+$, then the restriction of $\Phi$  to ${\bf C}^+$
 maps conformally 
${\bf C}^+$ to a domain inside ${\bf C}^+$. It follows that $$F_t(z):=\Psi(\Phi(z)+t)$$ is well defined for all $z\in {\bf C}^+$ and all $t\in \bf R$. Moreover these functions satisfy the equation
$$\frac{d}{dt}F_t(z)+\varphi(F_t(z))=0,\qquad F_0(z)=z, \quad \text{for $z\in{\bf C}^+$}$$
where   $\varphi=\psi\circ\Phi$ is a Nevanlinna function. One has  $F_t(F_s(z))=F_{t+s}(z)$ for all $s,t\in\bf R$ therefore
$F_{-t}=F_t^{-1}$. The maps $\Psi,\Phi$ conjugate the flow $F_t$ on $\Psi({\bf C}^+)$ with the horizontal translation flow on ${\bf C}^+$.
For every $t$ the function $\varphi\circ F^{-1}_t(z)=\psi(\Phi(z)-t)$ is a Nevanlinna function. 
Since $\varphi$ is a Nevanlinna function there exists $a=\lim_{y\to\infty}\frac{\varphi(iy)}{iy}\in ]-\infty,0]$. Assume that $a<0$, then $\Phi({\bf C}^+)$ is a domain included in a horizontal strip of height $-a\pi$ (this is easy to see if $\varphi$ is rational and follows in the general case by approximation). However it follows from the proof of Lemma \ref{NevC+}, parts (1) and (2), that $\Psi^{-1}({\bf C}^+)=\Phi({\bf C}^+)$ contains complex numbers with arbitrarily large imaginary parts.
We conclude that $a=0$
and the function $\varphi$ satisfies all the conditions in Theorem \ref{freenllevy}.

\subsubsection{}\label{converse}
Conversely, let $\varphi$ be a Nevalinna function satisfying the hypothesis of Theorem \ref{freenllevy}  and let $(F_t)_{t\geq 0}$ be the solution to 
$$\frac{d}{dt}F_t(z)+\varphi(F_t(z))=0,\quad F_0(z)=z.$$
Let $\Phi$ be a primitive of $-1/\varphi$ and $\Omega_0=\Phi({\bf C}^+)$. Since $-1/\varphi$ is a Nevanlinna function the function $\Phi$ is a conformal mapping $\Phi:{\bf C}^+\to \Omega_0$ with inverse
 $\Psi:\Omega_0\to{\bf C}^+$. The domain
 $-\Omega_0$ is starlike therefore  $\Omega_0+t\subset \Omega_0$ for all $t\geq 0$ moreover one has
$F_t(z)=\Psi(\Phi(z)+t)$ for all $z\in {\bf C}^+, t\geq 0$. For each $t>0$ the function $F_t$ is univalent  with inverse $F_t^{-1}$ defined at least in a  domain of the form $\Gamma_{\alpha,\beta}$ and one has 
$$\frac{d}{dt}F_t^{-1}(z)-\varphi(F^{-1}_t(z))=0$$ in this domain, therefore putting $F_{-t}=F^{-1}_t$ the equation
$\frac{d}{dt}F_t(z)+\varphi(F_t(z))=0$ also holds for $t<0$ at least in some domain $\Gamma_{\alpha,\beta}$.
The assumption is that, for all $t>0$, $\varphi\circ F_t^{-1}$ extends analytically to a Nevalinna function. 
For every $t$ one has
$$\frac{d}{dz}F_t(z)=\frac{\varphi(F_t(z))}{\varphi(z)}.$$  For $t<0$  the function 
$z\mapsto\varphi(F_t(z))$ has an analytic continuation to ${\bf C}^+$, therefore the functions  $\frac{d}{dz}F_t$ and $F_t$ also have such an analytic continuation.

If $z\in\Omega_0$ and $t\geq 0$ one has $\Psi(z)=F_{-t}(\Psi(z+t))$.
Let $\Omega=\cup_{t\geq 0}(\Omega_0-t)$, then  for every $z\in \Omega$ there exists $t\geq 0$ such that
$z+t\in\Omega_0$ therefore  $\Psi(z+t)\in{\bf C}^+$ and $F_{-t}(\Psi(z+t))$ is well defined, moreover it does not depend on $t$.
This gives an analytic continuation of $\Psi$ to $\Omega$, such that $\Psi'$ takes values with negative imaginary part. 
The domain $\Omega$ is stable under translation by real numbers. Since $\varphi(iy)/iy\to 0$ as $y\to\infty$ the function $\Phi$ takes values with arbitrarily high imaginary parts, therefore  $\Omega$ is either the whole complex plane, or a translate of the upper half plane. In the first case, since $\Psi$ is univalent, it must be a  polynomial of  degree $1$ and $\varphi$ is a constant with negative imaginary part. In the second case, since $\Phi$ was defined up to an integration constant, we can assume that $\Omega$ is equal to  ${\bf C}^+$ and  $\Psi$ is a univalent map whose derivative is $-\psi$ with $\psi$ a Nevalinna function, moreover  $\Psi({\bf C}^+)$ contains ${\bf C}^+$ so that we are in the situation of section \ref{direct}.

\subsection{}
Finally we can summarise the preceding results and state the free nonlinear L\'evy-Khinchine formula, which characterises the
Nevanlinna functions appearing in Theorem \ref{freenllevy}. 

\begin{theorem}\label{freelk}
Let $\varphi$ be a FAL2 function, then either $\varphi$ is a constant, or there exists a univalent function $\Psi$, with inverse $\Phi$ and  derivative  $\Psi'=-\psi$, where $\psi$ is a Nevanlina function, such that $\Psi({\bf C}^+)$ contains ${\bf C}^+$ and such that $\varphi=\psi\circ \Phi$. Conversely, for any functions $\psi,\Psi,\Phi$ satisfying the above requirements, the Nevanlinna function
$\varphi=\psi\circ\Phi$  is a FAL2 function.
\end{theorem}

As we see from the above theorem, the FAL2 functions can be parametrised by a convex  set, i.e. the functions $\Psi$, however going from
$\Psi$ to $\varphi=-\Psi'\circ\Psi^{-1}$ is a nonlinear map.
\section{The case of free multiplicative convolution}
In this section we consider the case of free multiplicative convolutions of measures, on the unit circle and on the positive half-line, recalling Theorems 3.5, 3.6, 4.6.1 and 4.6.2 of \cite{B} and giving the analogues of the nonlinear free L\'evy-Khinchine formula, Theorem \ref{freelk}. Since this is  very similar to the additive case, we only sketch the arguments.
\subsection{Free multiplicative convolution on the circle}
\subsubsection{} Let
$\mu$ and
$\nu$ be probability measures on the unit circle $\bf T$ and let $U$ and
$V$ be two unitary elements in some non-commutative probability space
$(A,\tau)$, with respective distributions $\mu$ and $\nu$, then the
distribution of $UV$ is called the free multiplicative convolution of
$\mu$ and $\nu$ and is denoted by $\mu\boxtimes \nu$. Define 
$$\eta_{\mu}(z)=\int_{\bf T}{z\xi\over 1-z\xi}d{\mu}(\xi)$$
 Let $\cM_*$ be the set of probability
measures on $\bf T$ such that $\int_{\bf T}\xi d\mu(\xi)\not=0$. If
$\mu\in\cM_*$ then the function ${\eta_{\mu}\over 1+\eta_{\mu}}$ has a right
inverse, called 
$\tilde{\chi}_{\mu}$, defined in a neighbourhood of $0$,  such that
$\tilde \chi_{\mu}(0)=0$, and we let
$\Sigma_{\mu}(z)= \frac{1}{z}\tilde\chi_{\mu}(z)$ be the
$\Sigma$-transform of $\mu$. Then, for any measures
$\mu,\nu\in\cM_*$,  one has $\mu\boxtimes \nu\in\cM_*$ and 
$$\Sigma_{\mu\boxtimes\nu}(z)= \Sigma_{\mu}(z)\Sigma_{\nu}(z)$$
in some neighbourhood of zero where these three functions are defined.
If one of the measures has zero mean then $\mu\boxtimes\nu$ is the uniform measure on $\T$.
\subsubsection{}
The analogue, for free multiplicative convolution on $\T$, of the
L\'evy-Khinchine formula, states that a probability measure on $\T$
is infinitely divisible, for the free multiplicative convolution,  if and
only
if its $\Sigma$ transform can be written as 
$\Sigma_{\mu}(z)=\exp(u(z))$ where $u$ is an analytic function on the open unit disk $\bf D$,
taking values with nonnegative real parts. Such a function has a  
representation of the form 
$$u(z)=i\alpha +\int_{\T}{1+\zeta z\over 1-\zeta z}d\nu (\zeta)$$
for some finite positive measure $\nu$ on $\bf T$, and real number
$\alpha$.
\subsubsection{}
The Markov and subordination property of the free multiplicative convolution is given by the following Theorem 3.5 from \cite{B}.
 \begin{theorem} \label{suboru}
Let
$(A,\tau)$ be a non commutative probability space, 
$B\subset A$ be a von Neumann subalgebra, and $U,V\in A$ such that $U$
and $V$ are unitary, with  respective distributions $\mu$ and
$\nu$, one has  $U\in B$, and $V$ is free with $B$, then 
there exists a Feller Markov kernel 
$\cK=k(\xi,d\omega)$ on
$\T\times \T$ and an analytic function $F$, defined on $\D$, such that
\begin{enumerate}
\item For any bounded Borel function $f$ on $S$, one has
$\tau(f(UV)\vert B)=\cK f(U).$
 \item $\vert F(z)\vert\leq\vert
z\vert$, for $z\in \D$.
\item for all $z\in \D$ one has
$\int_{\T}{z\omega\over 1-z\omega}k(\xi,d\omega)=\frac{F(z)\xi}{
1-F(z)\xi}$ 
\item  for all $z\in \D$ one has 
$\eta_{\mu}(F(z))=\eta_{\mu\boxtimes \nu}(z)$
\end{enumerate}
 If $\mu\in\cM_*$, the
map
$F$ is uniquely determined by the properties $(2)$ and $(4)$.
\end{theorem}
\subsubsection{}
Processes with unitary multiplicative free increments are defined analogously to the additive case, and such processes with homogeneous transition probabilities were called FUL2 processes, for which we now recall the analogue of Theorem \ref{freenllevy}.
\begin{theorem}\label{freemulevy}
 Let $\mu_t,t\geq 0$ and 
$(\mu_{s,t})_{s<t\in\bf R_+}$  be families of probability measures on $\T$ satisfying 
\begin{equation}\label{sg}
\mu_s\boxtimes\mu_{s,t}=\mu_t;\quad \mu_{s,t}\boxplus\mu_{t,u}=\mu_{s,u}
\end{equation}
for all $s<t<u$. Let 
$(\cK_{s,t})_{s<t\in\bf R_+}$ be the corresponding Markov transition functions on $\T$.
Assume that the kernels are time homogeneous, then 
 the kernels
$\cL_{t}\equiv\cK_{0,t}$ for $t\geq 0$, 
 form a Feller Markov
semi-group. Let $F_{s,t}$ be the analytic functions  associated to the
kernels $\cK_{s,t}$ by Theorem \ref{suboru}. The maps $F_{t}\equiv F_{0,t}$, where
$F_0$ is the identity function, form a continuous semigroup, under composition, of
analytic transformations of
${\D}$, moreover there exists a  
function $u$ on $\D$, taking values with nonnegative real part,  such
that 
the maps 
$F_t$, for $t\geq 0$, satisfy the differential equation 
\begin{equation}\label{diffu}
\frac{\partial F_t}{
\partial t}+F_tu(F_t)=0,\qquad F_0(z)=z.
\end{equation}
Conversely,  let  $u$ be an   analytic function 
on
$\D$, such that
$\Re(u(z))\geq 0$ for all $z\in \D$,
 and let
$F_t$, for $t\geq 0$, be the solution of the differential equation
$\frac{\partial
F_t}{
\partial t}+F_tu(F_t)=0$, with $F_0(z)=z$, then there exists
 a free multiplicative L\'evy process of the second kind, with initial
distribution
$\mu_0$, with associated semi-group of maps $(F_t)_{t\in\R_+}$, if and
only if, for every
$t>0$ the function
$u\circ F_t^{-1}\circ {\tilde{\chi}}_{\mu_0}^{-1}$ has an analytic
continuation to
$\D$, taking values with nonnegative real
part.

\end{theorem}
Functions like $u$ in the above theorem are called FUL2 L\'evy functions in \cite{B}.
In the following we consider the case where $\mu_0=\delta_1$.

Let $u$ be a FUL2 function. Let us change variables and put $z=e^{iw}$ with $\Im(w)>0$. Then the function $\tilde u(w)=-iu(e^{iw})$ is a Nevanlinna function which is periodic of period $2\pi$, and the differential equation (\ref{diffu}) becomes  
$\frac{\partial
\tilde F_t}{
\partial t}+\tilde F_t\tilde u(\tilde F_t)=0$, with $\tilde F_0(w)=w$. One has $\tilde F_t(w+2\pi)=\tilde F_t(w)+2\pi$. Introducing a primitive of $-1/\tilde u$ and reasoning as above we see that there must exist a $2\pi$-periodic Nevanlinna function $\psi$, with $\Psi$ a primitive of $-\psi$ such that $\Psi(\C^+)$ contains $\C^+$, however if $\psi$ is $2\pi$ periodic then the measure $(1+u^2)\nu(du)$ is also $2\pi$ periodic therefore the integral $\int_0^\infty (1+u^2)\nu(du)=\infty$, unless $\psi$ is constant.
We conclude:
\begin{theorem}
All FUL2 functions are constant.
\end{theorem}

\subsection{Multiplicative free convolution on the positive half-line}
\subsubsection{} Let 
$\mu$ be a probability measure on
$\bf R_+$, different from
$\delta_0$, and define 
$$\eta_{\mu}(z)=\int_{\R_+}{z\xi\over 1-z\xi}d{\mu}(\xi)$$ 
This function is analytic on $\C\setminus \bf R_+$, and $\eta_{\mu}(\bar
z)=\bar\eta_{\mu}(z)$ for $z\in \C\setminus \bf R_+$. The function 
 ${\eta_{\mu}\over 1+\psi_{\mu}}$ is univalent on $i\bf C^+$, its image
contains a neighbourhood of the interval $]\mu(\{0\})-1,0[$ in $\C$.  Let 
$\tilde{\chi}_{\mu}$ be the right inverse of this function on the image 
${\eta_{\mu}\over 1+\eta_{\mu}}(i\bf C^+)$. We define the
$\Sigma$-transform of $\mu$ as the function 
$\Sigma_{\mu}(z)= \frac{1}{z}\tilde\chi_{\mu}(z)$ defined on 
${\eta_{\mu}\over 1+\eta_{\mu}}(i\C^+)$.

 Let $\mu$ and $\nu$ be two probability measures on $\bf R_+$,
different from $\delta_0$ and let $S$ and $T$ be free random variables, in
some non-commutative probability space, with respective distributions
$\mu$ and $\nu$, then the distribution of $S^{1\over 2}TS^{1\over 2}$ which
is also the distribution of $T^{1\over 2 }ST^{1\over 2}$, is the free
multiplicative convolution of $\mu$ and $\nu$, denoted by $\mu\boxtimes
\nu$, and one has $\Sigma_{\mu\boxtimes
\nu}=\Sigma_{\mu}\Sigma_{\nu}$ on some neighbourhood of the interval
$]-\varepsilon,0[$, for some $\varepsilon>0$.

\subsubsection{}
The Markov and subordination property of the free multiplicative convolution on the positive half-line is given by the following Theorem 3.6 from \cite{B}.
\begin{theorem}\label{suborp}
Let
$(A,\tau)$ be a non commutative probability space, 
$B\subset A$ be a von Neumann subalgebra, and $S,T\in \tilde A_{sa}$ such
that
$S$ and $T$ are positive, with  respective distributions $\mu$ and
$\nu$, different from $\delta_0$, one has  $S\in \tilde B_{sa}$ and  $T$ is
free with
$B$, then 
there exists a Feller Markov kernel 
$\cK=k(u,dv)$ on
$\R_+\times \R_+$ and an analytic function $F$, defined on $\C\setminus
\R_+$, such that
\begin{enumerate}
\item for any bounded Borel function $f$ on $S$ one has
$\tau(f(S^{1/2}TS^{1/2})\vert B)=\cK f(S).$
 \item 
 $F(\zeta)\in\C^+$, $F(\bar \zeta)=\bar F(\zeta)$ and $\Arg\,
(F(\zeta))\geq\Arg\, (\zeta)$ for 
$\zeta\in
\C^+$.
\item for all $\zeta\in\C^+$ one has
$\int_{\R_+}{\zeta v\over 1-\zeta v}k(u,dv)={F(\zeta)u\over_
1-F(\zeta)u}.$ 
\item  For all $\zeta\in\C^+$ one has
$\eta_{\mu}(F(\zeta))=\eta_{\mu\boxtimes
\nu}(\zeta).$
\end{enumerate}
The
map
$F$ is uniquely determined by the properties $(2)$ and $(4)$.
\end{theorem}

\subsubsection{}
 Again one has a L\'evy-Khinchine formula, where freely  infinitely
divisible probability measures on $\bf R_+$ are characterised as having 
$\Sigma$-transforms of the form 
$\Sigma_{\mu}(z)=\exp(v(z))$, where $v$ is an analytic function on
$\C\setminus\R_+$, with $v(\bar z)=\bar v(z)$, and 
$v(\bf C^+)\subset\bf C^-\cup\bf R$. Such functions have the representation 
$$v(z)=az+b+\int_0^{+\infty}\frac{1+tz}{ z-t}d\nu(t)$$ for some
real numbers $a\leq 0$ and $b$, and $\nu$ a finite  positive
measure on
$\R_+$.
The analogue of Theorems  \ref{freenllevy} and  \ref{freemulevy} is the following.
\begin{theorem}
 Let $\mu_t,t\geq 0$ and 
$(\mu_{s,t})_{s<t\in\bf R_+}$  be families of probability measures on $\bf R_+$ satisfying 
\begin{equation}\label{sg}
\mu_s\boxtimes\mu_{s,t}=\mu_t;\quad \mu_{s,t}\boxplus\mu_{t,u}=\mu_{s,u}
\end{equation}
for all $s<t<u$. Let 
$(\cK_{s,t})_{s<t\in\bf R_+}$ be the corresponding Markov transition functions on $\bf R_+$.
Assume that the kernels are time homogeneous, then 
 the kernels
$\cL_{t}\equiv\cK_{0,t}$ for $t\geq 0$, 
 form a Feller Markov
semi-group. Let $F_{s,t}$ be the analytic functions  associated to the
kernels $\cK_{s,t}$ by Theorem \ref{suborp}.  The maps $F_t\equiv
F_{0,t}$ for 
$t\geq 0$, form a semigroup of analytic maps on $\C\setminus\R_+$, such
that
 $t\mapsto \Arg \, F_t(z)$ is an increasing  map for $z\in\C^+$. There
exists an   analytic function $v$ on $\C\setminus\R_+$, 
$\C^-\cup\R$, 
such that $v(\bar z)=\bar v(z)$ for $z\in\C^+$, $v(\C^+)\subset
\C^-\cup\R$,  and the maps 
$F_t$, for $t\geq 0$, satisfy the differential equation ${\partial
F_t\over
\partial t}+F_tv(F_t)=0$.\item  Let  $v$ be an   analytic function 
 on $\C\setminus \R_+$, such that $v(\C^+)\subset\C^-\cup\R$, and $v(\bar
z)=\bar v(z)$ for all 
$z\in\C^+$, 
 and let
$F_t$, for $t\geq 0$, be the solution of the differential equation
${\partial
F_t\over
\partial t}+F_tu(F_t)=0$, with $F_0(z)=z$, then there exists
 a free multiplicative L\'evy process of the second kind, with initial
distribution
$\mu_0$, with associated semi-group of maps $(F_t)_{t\in\R_+}$, if and
only if, for every
$t>0$ the function
$v\circ F_t^{-1}\circ {\tilde{\chi}}_{\mu_0}^{-1}$ has an analytic
continuation to
$\C\setminus\R_+$, such that $v(\bar z)=\bar v(z)$, and
$v(\C^+)\subset\C^-\cup\R$. 
\end{theorem}

\subsubsection{}
We now determine all FPL2 functions, in the case $\mu_0=\delta_1$.

We change variables and put $z=-e^w$ where $z\in\C\setminus\R_+$ and
 $w\in \mathcal S$ where $\mathcal S$ is the symmetric horizontal strip
$\mathcal S=\{w|\Im(w)\in]-\pi,\pi[\}$. Let $v$ be a FPL2 function and define $\tilde v(w)=v(-e^{w})$.
Then $v$ is analytic in the strip $\mathcal S$, satisfies $v(\bar w)=\bar v(w)$ and takes values with positive imaginary part on $\mathcal S\cap \C^+$. With $F_t(-e^w):=-\exp(\tilde F_t(w))$ the equation ${\partial
F_t(z)\over
\partial t}+F_tv(F_t(z))=0$ becomes  ${\partial
\tilde F_t(w)\over
\partial t}+\tilde v(\tilde F_t(w))=0$.
The function $\tilde v$ has at most one zero $\omega_0$, on the real line.
 Let $\tilde V(w)$ be a primitive of $-1/\tilde v$ on $\mathcal S\setminus ]-\infty,\omega_0]$ such that  $\tilde V$ takes real values on $]\omega_0,+\infty]$. One has $\tilde V(\bar w)=\bar{\tilde V}(w)$, moreover   $\tilde V$ is univalent on $\mathcal S\setminus ]-\infty\omega_0]$ and the domain $\Omega=\tilde V(\mathcal S)$ satisfies $\bar \Omega=\Omega$ and $\Omega+t\subset\Omega$
for $t\geq 0$.
Let $\tilde W$ be the inverse of $\tilde V$ then one has $\tilde F_t(w)=\tilde W(\tilde V(w)+t)$. As in section \ref{converse} one can extend the map $\tilde W$ to a univalent function on the domain $\Omega_\infty=\cup_{t\in\bf R}(\Omega+t)$ which is either the whole complex plane or a horizontal strip, symmetric with respect to the real axis. The function $\tilde W$ satisfies $\bar{\tilde W}(z)=\tilde W(\bar z)$ moreover 
$\Im(\tilde W'(w))\leq 0$ for $\Im(w)>0$. From these considerations we deduce the analogue of Theorem \ref{freelk}:

\begin{theorem}
Let $v$ be a FPL2 function, then either $v$ is a constant, or there exists a univalent function $U$ defined on a symmetric horizontal strip $\mathcal T$,  such that $U(\mathcal T)$ contains $\mathcal S$, one has $U(\bar z)=\bar U(z)$, and  $U'=-u$, with $\Im(u(z))>0$ for $\Im(z)>0$, and $v=u\circ U^{-1}$. Conversely, for any function $U$ satisfying the above requirements, the function
$v=u\circ U^{-1}$  is a FPL2 function.
\end{theorem}

\end{document}